\documentclass{article}

\usepackage{amssymb, latexsym,pdfsync,amsmath,amsthm, ulem,hyperref,graphicx}
\usepackage{fullpage}
\usepackage{pgf,tikz}
\usepackage{mathrsfs}
\usepackage{mathtools}
\usetikzlibrary{arrows}
\newtheorem{theorem}{Theorem}
\newtheorem{corollary}{Corollary}
\newtheorem{lemma}{Lemma}

\theoremstyle{definition}
\newtheorem{definition}{Definition}

\newtheorem{example}{Example}

\newcommand{\SSS}{\mathbb{S}}

\newcommand{\FF}{\mathbb{F}}

\newcommand{\Fq}{\mathbb{F}_q}

\newcommand{\Fqm}{\mathbb{F}_{q^m}}

\def\Fq{{\mathbb{F}}_q}
\def\Fqs{{\mathbb{F}}_{q^s}}

\def\im{\mathrm{im}}

\def\PG{\mathrm{PG}}

\def\dim{\mathrm{dim}}

\newcommand{\npmatrix}[1]{\left( \begin{matrix} #1 \end{matrix} \right)}

\newcommand{\rank}{\mathrm{rank}}

\begin{document}
\title{New invariants for rank metric codes, with 
 applications to the classification of rank two semifields of order 256}
\author{Jack Gilchrist, Stefano Lia, Arani Paul, and John Sheekey}
\maketitle

\begin{abstract}
    In this paper we completely classify semifields of order $2^8=256$ containing a nucleus of order $2^4=16$. We introduce new invariants for semifields, and apply new computational techniques for calculating old invariants. Together these make the computational classification significantly quicker.
\end{abstract}

\section{Introduction}

A {\it finite semifield} is a finite division algebra in which multiplication is not necessarily associative. Thus they are natural generalisations of finite fields. Dickson \cite{Dickson1906} showed that there exist non-trivial finite semifields. Albert \cite{Albert1960} and Knuth \cite{Knuth1965} deepened the study of semifields, demonstrating equivalences between semifields and certain types of projective planes and hypercubes (threefold tensors). Since then semifields have proved useful in a variety of settings; we refer to \cite{LaPo} for further background and history.

It is well known that to each semifield $\SSS$ one can associate a {\it semifield spread set} $C(\SSS)$; an additively closed set of matrices in which the difference of any pair of matrices is invertible. By an appropriate choice of parameters, a semifield spread set is an $\Fq$-subspace of $M_n(\Fqs)$ of dimension $ns$ over $\Fq$ in which every nonzero element is invertible. In this setting we can view $C(\SSS)$ as a {\it rank metric code}, specifically an {\it MRD code} in $M_n(\Fqs)$ with minimum distance $n$. Semifields which give rise to spread sets in this space are those who have {\it right nucleus} containing $\Fqs$. 

The natural notion of equivalence for semifields is that of {\it isotopism}; this notion coincides with the equivalence of the corresponding spread sets as matrix rank metric codes. This fact has been exploited in order to perform computational classifications of semifields of small orders; see \cite{Rua2009}, \cite{Rua2011b}, \cite{Rua2012a}, \cite{Rua2012b}. 

Further advances for general semifields of larger orders appears to be out of reach at present. Therefore attention turns to semifields with added assumptions; for example, symplectic semifields \cite{LaShSymp}, \cite{MaPe}, or semifields of small dimension over a nucleus \cite{MaPoNuc}. When both these assumptions are combined we consider {\it rank two commutative semifields}, which have numerous connections to other topics in finite geometry. See \cite{BaBlLark2} for these connections, and a remarkable classification result.

Strong classification results on semifields of dimension $2$ over a nucleus $\Fqs$ have been achieved for for arbitrary $q$ for $s=2$ in \cite{CaPoTr} and $s=3$ in a series of papers including \cite{MaPoTr}, \cite{JoMaPoTr}. 

In this paper we introduce some new invariants for semifields and rank-metric codes which can be used to more efficiently determine inequivalence. We demonstrate their utility by classifying semifields of order $2^8$ containing a nucleus of order $2^4$, or equivalently MRD codes in $M_2(\FF_{2^4})$ with minimum distance $2$.

\section{Rank metric codes and spread sets}

\subsection{Matrix rank metric codes}

For the purposes of this paper, a {\it matrix rank metric code} is an $\Fq$-subspace of $M_{\ell \times n}(\Fqs)$, equipped with the distance function 
\[
d(A,B) := \rank(A-B).
\]
We will mostly work with the case $\ell=n$, and write $M_n(\Fqs) := M_{\ell\times n}(\Fqs)$.

Two codes $C,C'$ are said to be {\it equivalent} if there exist invertible matrices $X,Y\in M_n(\Fqs)$ and a field automorphism $\rho$ of $\Fqs$ such that
\[
C' = \{XA^\rho Y:A\in C\},
\]
where $A^\rho$ is the matrix obtained from $A$ by applying the automorphism $\rho$ entrywise. Throughout we will let $\sigma$ denote the Frobenius automorphism, that is, the map $\sigma:x\rightarrow x^q$. 

Each matrix of $M_n(\Fqs)$ naturally defines an $\Fqs$-endomorphism from $(\Fqs)^n$ to itself. We can also regard such a map as an $\Fq$-endomorphism from $(\Fq)^{ns}$ to itself; we will make this explicit later.

The problem of determining the equivalence or otherwise of two codes is well known to be a difficult problem; indeed, many public key cryptosystems rely on the computational hardness of this problem. Therefore any {\it invariant}, that is, any (easily computable) function from the set of codes to some target set which takes the same value at any two equivalent codes, holds great value in both cryptographic and classification problems.

\subsection{Vector rank metric codes and $q$-systems}\label{ssec:qsys}

A {\it vector rank metric code} is an $\Fqs$-subspace of $(\Fqs)^N$ equipped with the distance function
\[
d(u,v) := \dim_{\Fq}\langle u_i-v_i:i\rangle_{\Fq}.
\]
By picking an $\Fq$-basis for $\Fqs$, we can identify such a code with an $\Fq$-subspace of $M_{s\times N}(\Fq)$, in which case the distance function becomes the rank of the difference of two matrices, as for matrix rank metric codes.

Two vector rank metric codes $D,E$ are said to be {\it equivalent} if there exists an invertible matrix $Q\in M_N(\Fq)$ and a field automorphism $\rho$ of $\Fqs$ such that
\[
E = \{v^\rho Q:v\in D\}.
\]
It is much easier to determine the equivalence or otherwise of two vector rank metric codes; see for example \cite{couvreur2020hardness}. 

A correspondence between vector rank metric codes and {\it $q$-systems} has been developed in recent years, leading to fruitful rewards in both areas; see for example \cite{alfarano2022linear}, \cite{randrianarisoa2020geometric}, \cite{marino2023evasive}. A $q$-system is an $\Fq$-subspace of an $\Fqs$-vector space $(\Fqs)^k$, with equivalence defined under the natural action of $\Gamma L(k,\Fqs)$. Such objects have been studied for the past couple of decades in finite geometry through the {\it linear sets} they define in $\PG(k-1,q^s)$. The correspondence is realised by taking a basis for $D$, creating a matrix $G$ whose rows are this basis (a {\it generator matrix} for $D$), and associating to $G$ the $\Fq$-subspace $U$ generated by its columns. Conversely, choosing an $\Fq$-basis for $U$ and forming a matrix $G$ with these basis elements as its columns, we obtain a vector rank metric code as the $\Fqs$-subspace generated by its rows. While some care should be taken around degeneracy (where the dimensions of either $D$ or $U$ may not equal $k$ or $N$ respectively), in general we can pass freely between these two settings.

Consider an $\Fq$-subspace $C$ of $M_{\ell\times n}(\Fqs)$ of $\Fq$-dimension $k$. Then we can define a vector rank metric code $D$ in $(\Fqs)^k$ of dimension at most $\ell n$ as follows. We identify $M_{\ell\times n}(\Fqs)$ with $(\Fqs)^{\ell n}$ as $\Fq$-vector spaces, for example by identifying a matrix $A$ with its vectorisation $\mathrm{vec}(A)$ in column-major order. Then the image of $C$ is naturally a $q$-system, which defines by the above correspondence a vector rank metric code of length $k$ and dimension at most $\ell n$.

\subsection{Semifield spread sets}

Let $\SSS$ be a presemifield with multiplication $\star$. Then the maps
\begin{align*}
L_x&:y\mapsto x\star y\\
R_y&:x\mapsto x\star y
\end{align*}
are known as the {\it maps of  left-} and {\it right-multiplication} respectively. Each are additive, and invertible for nonzero $x,y$.

\begin{definition}
The {\it spread set} of a semifield $\SSS$ is denoted by $C(\SSS)$ and defined as 
\[
\{R_y:y\in \SSS\}.
\]
\end{definition}

It is well-known that there exist fields $\Fq$, $\Fqs$, and a positive integer $n$ such that $C(\SSS)$ can be naturally embedded into $M_n(\Fqs)$ as an $\Fq$-subspace in which every nonzero element is invertible. We refer to \cite{Willems} for a detailed exposition. In this way we see $C(\SSS)$ as an $\Fq$-linear MRD code in $M_n(\Fqs)$.

\section{Vector rank-metric codes from spread sets}\label{sec:vmrss}

As outlined in the previous section, we may construct a matrix $G$ from an $\Fq$-basis of $C(\SSS)$ to view $C(\SSS)$ as a $q$-system, and hence obtain an $\Fqs$-linear vector rank metric code $D(\SSS)$ of dimension at most $n^2$ and length $ns$. Choosing different bases for $C(\SSS)$ leads to equivalent codes, so we abuse notation and omit this choice.

\begin{lemma}
    If two semifield spread sets $C(\SSS)$ and $C(\SSS')$ are equivalent, then the vector rank-metric codes $D(\SSS)$ and $D(\SSS')$ are equivalent.
\end{lemma}

\begin{proof}
    Suppose $C(\SSS')=XC(\SSS)^\rho Y$ for some $X,Y\in M_n(\Fqs)$ invertible. Choose an $\Fq$-basis $\{A_1,\ldots,A_{ns}\}$ of $C(\SSS)$, and form the $\Fq$-basis $\{XA_1^\rho Y,\ldots,XA_{ns}^\rho Y\}$ of $C(\SSS')$. Then $\mathrm{vec}(XA_i^\rho Y) = (Y^T\otimes X)\mathrm{vec}(A_i)^\rho$, and so the matrices $G(\SSS)$ and $G(\SSS')$ satisfy $G(\SSS') = (Y^T\otimes X)G(\SSS)^\rho$, giving that $D(\SSS)$ and $D(\SSS')$ are equivalent. 
\end{proof}

Note that the converse is certainly not true in general; two inequivalent spread sets can certainly give rise to equivalent vector rank-metric codes. However we can utilise the contrapositive of this lemma to act as an invariant of semifields.

\begin{corollary}\label{cor:vrm}
    If the vector rank-metric codes $D(\SSS)$ and $D(\SSS')$ arising from two semifields $\SSS$ and $\SSS$ are inequivalent, then the semifield spread sets $C(\SSS)$ and $C(\SSS')$ are inequivalent, and the semifields $\SSS$ and $\SSS'$ are not isotopic.
\end{corollary}

This approach was in part used in \cite{CaPoTr}, \cite{MaPoTr} to assist in the classification of semifields with $n=2,s=2,3$. There the methods were stated in terms of linear sets rather than $q$-systems, and certain geometric properties of the $q$-systems were used rather than the full equivalence of the vector rank-metric codes.





\section{Embeddings of matrix spaces}\label{sec:embed}

It is well-known that we can view matrices over a large field as larger matrices over a subfield; that is, we can create the following embedding.
\[
M_n(\Fqs)\hookrightarrow M_{ns}(\Fq)\hookrightarrow M_{ns}(\Fqm).
\]
The first embedding can be achieved by replacing elements of $\Fqs$ by their corresponding element of $M_s(\Fq)$ under a fixed regular representation, for example by mapping a primitive element $\theta$ to the companion matrix of its minimal polynomial. Concretely, let $\phi:\Fqs\rightarrow M_s(\Fq)$ be such a regular representation. Then we define $\overline{\phi}:M_n(\Fqs)\rightarrow M_{ns}(\Fq)$ by mapping the matrix $A = (a_{ij})$ to the block matrix $\overline{\phi}(A) := (\phi(a_{ij}))$. The second embedding is simply extending scalars to $\Fqm$, where $m$ is any positive integer.

\begin{definition}
    Let $U$ be an $\Fq$-subspace of $M_{ns}(\Fq)$. We define the {\it $m$-th embedded space} \[
    E_m(U)=\langle U\rangle_{\Fqm}.
    \]
    For an $\Fq$-subspace $C$ of $M_n(\Fqs)$, we define 
    \[
    E_m(C):= E_m(\overline{\phi}(C)).
    \]
\end{definition}

Let $\SSS$ be a semifield of order $q^{ns}$ with a nucleus of order a power of $q^s$, and $C(\SSS) \subset M_n(\Fqs)$ its spread set. The following is immediate from the definition of equivalence.

\begin{theorem}
    Suppose $\SSS$ and $\SSS'$ are isotopic. Then $E_m(C(\SSS))$ and $E_m(C(\SSS'))$ are equivalent for any $m$, and hence have the same rank distribution.
\end{theorem}

In this way we can exploit ideas from (matrix) rank-metric codes which are not useful when studying $C(\SSS)$ directly. The penalty is that we must consider a larger space, which obviously comes with a computational cost.

\begin{definition}
    Let \(\SSS\) be a semifield of order \(q^{ns}\) with $C(\SSS)$ an $\Fq$-subspace of $M_n(\Fqs)$, and let \(m \in \mathbb{N}\). We define the {\it \(m\)-ranks of \(\SSS\)}, \(\mathrm{ranks}^{(m)}(\SSS)\), as the multiset,
    \[
    \mathrm{rank}^{(m)}(\SSS) \coloneqq \{ \rank(A) : A \in E_m(C(\SSS)) \}.
    \]
\end{definition}
We note that the \(1\)-ranks of a semifield is the same for all semifields, as each non-zero element of \(E_1(\SSS)\) has full rank. Hence $\mathrm{rank}^{(1)}(\SSS)= \{0^1,(ns)^{q^{ns}-1}\}$, where the exponents are the multiplicities.

\begin{corollary}\label{cor:embed}
     If $\mathrm{rank}^{(m)}(\SSS)\ne \mathrm{rank}^{(m)}(\SSS')$ for some $m$, then the semifield spread sets $C(\SSS)$ and $C(\SSS')$ are inequivalent, and the semifields $\SSS$ and $\SSS'$ are not isotopic.
\end{corollary}

\section{Classifying semifields}\label{sec:class}

We will use Corollaries \ref{cor:vrm} and \ref{cor:embed} in order to speed up the classification of semifields of order $q^{ns}$ with nucleus containing $\Fqs$. We will assume an equivalence testing algorithm, for example as in \cite{Rua2009}. Our advantage comes from reducing the number of times that this algorithm has to be performed.

Let us fix some $\theta$ such that $\Fqs=\Fq(\theta)$.  We form $ns$ sets $S_{i+js}$ of invertible matrices whose first row is $ \theta^ie_j$ for $i=0,\ldots,s-1$ and $j=1,\ldots,n$. Note that any semifield spread set in $M_n(\Fqs)$ must contain a basis containing precisely one element from each $S_{k}$.

For an $\Fq$-subspace $C$ of $M_n(\Fqs)$, we define $i(C)= (\rank^{(2)}(\overline{\phi}(C)),D(C))$, where $D(C)$ is a representative for the equivalence class of the vector rank-metric code obtained from $C$ through the process outlined in Section \ref{ssec:qsys}.

We proceed iteratively as in \cite{Rua2009}. Suppose we have fully classified up to equivalence all $k$-dimensional $\Fq$-subspaces of $M_n(\Fqs)$ spanned by elements of $S_1,\ldots,S_k$ in which every nonzero element is invertible; let $T_k$ denote a set of representatives. Let $T_{k+1}$ and $i_{k+1}$ be the empty set.

\begin{itemize}
    \item For $C'\in T_k$, $A\in S_{k+1}$, construct $C= \langle C',A\rangle_{\Fq}$. If every nonzero element of $C$ is invertible, continue.
\item Calculate $i(C)$. 
\begin{itemize}
    \item If $i(C)\notin i_{k+1}$, add $C$ to $T_{k+1}$ and $i(C)$ to $i_{k+1}$.
    \item If $i(C)\in i_{k+1}$, sequentially test if $C$ is equivalent to $C''$ for all $C''\in T_{k+1}$ such that $i(C'')=i(C)$. If each of these tests returns false, add $C$ to $T_{k+1}$.
\end{itemize} 
\end{itemize}
Repeat this until $k=ns$. We obtain $T_{ns}$ a complete set of representatives for the equivalence classes of MRD codes in $M_n(\Fqs)$, and hence a complete set of representatives for the isotopy classes of semifields of order $q^{ns}$ with right nucleus containing $\Fqs$.

\subsection{Classification for $q=2$, $n=2$, $s=4$}

We implemented the above procedure for $M_2(\FF_{2^4})$. 

\begin{theorem}
    The number of isotopy classes of finite semifields of order $2^8$ with right nucleus containing $\FF_{2^4}$, equivalently the number of equivalence classes of MRD codes in $M_2(\FF_{2^4})$ of minimum distance $2$, is 757. 
\end{theorem}

We outline now to what extent the invariants introduced here speed up the calculation.

Amongst the $757$ classes there are
\begin{itemize}
    \item  $66$ different $2$-ranks;
    \item $17$ different equivalence classes of vector rank-metric codes;
    \item $378$ different pairs $i(C)$.
\end{itemize}

At the final step, there were $530873$ spread sets. In order to classify them without utilising the invariants, we would need to calculate roughly $2\times 10^8$ equivalence tests, assuming they are uniformly distributed. With a version of \cite{Rua2009} tailored to this case, such a computation would take about $50$ days.

Instead, we calculate $2\times 10^6$ invariants, and then many orders of magnitude fewer equivalence tests.

A direct equivalence test takes approx $0.021$ seconds; calculating the $2$-ranks takes approx $0.08$ seconds, and calculating the equivalence class of vector rank-metric codes takes approx $0.03$ seconds. Overall, this step of the calculation is performed approximately $20$ times faster using these invariants than without using them. As the invariants can be used at each step in creating $T_k$, we obtain significant time savings.

\section{Divisibility properties of embedded spaces}

In this section we will demonstrate some divisibility properties of the ranks of embedded spaces. We extend the embedding from the previous section by extending scalars from $\Fq$ to $\Fqm$ to $\FF_{q^{\mathrm{lcm}(m,s)}}$ to obtain the following.
\[
M_n(\Fqs)\hookrightarrow M_{ns}(\Fq)\hookrightarrow M_{ns}(\Fqm)\hookrightarrow M_{ns}(\FF_{q^{\mathrm{lcm}(m,s)}}).
\]
Let $\phi$ be the regular representation from the previous section. 
It is well-known that we can simultaneously diagonalise the elements of $\im(\phi)$ over $\Fqs$. Explicitly, let $Z$ be an invertible {\it Moore matrix}; that is, 
\[
Z = \npmatrix{v_1&v_2&\cdots&v_s\\v_1^\sigma&v_2^\sigma&\cdots&v_s^\sigma\\\vdots&\vdots&\ddots&\vdots\\v_1^{\sigma^{s-1}}&v_2^{\sigma^{s-1}}&\cdots&v_s^{\sigma^{s-1}}}
\]
for some appropriately chosen $\Fq$-basis $\{v_1,\ldots,v_s\}$ of $\Fqs$. Then 
\[
Z\phi(\alpha)Z^{-1} = \mathrm{diag}(\alpha,\alpha^q,\ldots,\alpha^{q^{s-1}}) =:D_{\alpha x}.
\]
Note that $Z\phi(\alpha)Z^{-1}$ is the {\it Dickson matrix} corresponding to the {\it linearised polynomial} $f_\alpha(x) = \alpha x$. See for example \cite{SheekeyMRDsurvey} for details.

We can then perform a similar transformation for the image of $\overline{\phi}$ by defining $\overline{Z}= I_n\otimes Z$, and noting that
\[
\overline{Z}
\overline{\phi}(A)\overline{Z}^{-1} = (D_{a_{ij}x}). 
\]
We can then permute rows and columns using an appropriate permutation matrix $P$ so that

\[
(P\overline{Z})
\overline{\phi}(A)(P\overline{Z})^{-1} = A\oplus A^\sigma\oplus\cdots\oplus A^{\sigma^{s-1}}=:\psi(A),
\]
where $A^{\sigma^i}$ is the matrix obtained by raising each entry of $A$ to the power of $q^i$. Note that this resulting matrix $\psi(A)$ is an element of $M_{ns}(\Fqs)$. 

Let $C$ be any $\Fq$-subspace of $M_n(\Fqs)$. Define $\overline{\phi}(C)$ and $\psi(C)$ in the natural way. Then the following is straightforward.

\begin{lemma}
    For any $\Fq$-subspace $C$ of $M_n(\Fqs)$, the following multiset equalities hold:
    \[
    \{s\cdot\rank(x):x\in C\}= \{\rank(x):x\in \overline{\phi}(C)\}= \{\rank(x):x \in \psi(C)\}.
    \]
\end{lemma}

\begin{proof}
Since $P$ and $\overline{Z}$ are invertible, it is clear that $\rank(\overline{\phi}(A))= \rank(\psi(A))$. From the direct sum expression for $\psi(A)$, it is clear that $\rank(\psi(A))=s\cdot \rank(A)$, since the operations $A\mapsto A^{\sigma}$ preserve the rank function. \end{proof}

We have defined new invariants for $\Fq$-subspaces of matrices by considering the $\Fqm$-span of the elements of $\overline{\phi}(A)$. The previous discussion allows us to calculate and analyse the ranks of these elements of $M_{ns}(\Fqm)$ directly from elements of $M_n(\FF_{q^{\mathrm{lcm}(m,s)}})$.

\begin{lemma}
    Let $C$ be an $\Fq$-subspace of $M_n(\Fqs)$, with $\Fq$-basis $\{A_1,\ldots,A_k\}$. Let $\alpha_1,\ldots,\alpha_k\in \Fqm$. Then
\[
\rank\left(\sum_{i=1}^k \alpha_i\overline{\phi}(A_i)\right) = \sum_{j=0}^{s-1}\rank\left( \sum_{i=1}^k \alpha_iA_i^{\sigma^j} \right)= \sum_{j=0}^{s-1}\rank\left( \sum_{i=1}^k \alpha_i^{q^{m-j}}A_i \right).
\]
    If $m$ divides $s$, then 
    \[
\rank\left(\sum_{i=1}^k \alpha_i\overline{\phi}(A_i)\right) = \frac{s}{m}\sum_{j=0}^{m-1}\rank\left( \sum_{i=1}^k \alpha_i^{q^j}A_i \right).
\]
In particular, each element of the $\Fqm$-span of $\overline{\phi}(C)$ has rank divisible by $s/m$.
\end{lemma}

\begin{proof}
    The first equality holds simply by noting again that conjugating by $P\overline{Z}$ does not change the rank of a matrix, regardless of in which field the entries of the matrix lie. The second equality follows from the observation that for any $\alpha_i\in\Fqm$ it holds that $(\sum \alpha_i^{q^{m-j}}A_i)^{\sigma^j}= \sum\alpha_i A_i^{\sigma^m}$, and hence the ranks of $\sum \alpha_i^{q^{m-j}} A_i$ and $\sum\alpha_i A_i^{\sigma^m}$ are equal. The final equation then follows from the fact that $\alpha_i^{q^{rm}}= \alpha_i$ for each $i$ and each integer $r$, and a reindexing.
\end{proof}

\begin{example}
    In the computation of Section \ref{sec:class}, the $2$-ranks consist of even numbers, since $s/m=4/2=2$. In fact, we observe that the ranks of the nonzero elements of $E_2(C(\SSS))$ are always $4,6,$ or $8$. There exist some semifields where $E_2(C(\SSS))$ has no elements of rank $6$, and some where $E_2(C(\SSS))$ has no elements of rank $4$, but none where all nonzero elements have rank $8$.
\end{example}







\bibliographystyle{abbrv}
\bibliography{waifi_js.bib}

\end{document}